\documentclass{amsart}

\usepackage{amssymb}
\usepackage{times}
\newtheorem{theorem}{Theorem}[section]
\newtheorem{lemma}[theorem]{Lemma}
\newtheorem{corollary}[theorem]{Corollary}
\newtheorem{proposition}[theorem]{Proposition}

\numberwithin{equation}{section}

\newcommand{\RR}[0]{\mathbb{R}}

\newcommand{\pd}[2]{\frac{\partial #1}{\partial#2}}

\newcommand{\pdt}[0]{\frac{\partial}{\partial t}}
\newcommand{\delb}[0]{\overline{\nabla}}

\newcommand{\Rc}[0]{\operatorname{Rc}}
\newcommand{\Rm}[0]{\operatorname{Rm}}
\newcommand{\End}[0]{\operatorname{End}}
\newcommand{\Hc}[0]{\mathcal{H}}

\newcommand{\Hol}[0]{\operatorname{Hol}}
\newcommand{\hol}[0]{\mathfrak{hol}}

\begin{document}

\title[Full holonomy group under Ricci flow]{The full holonomy group under the Ricci flow}

\author[M. Cook]{Mary Cook}
\address{Arizona State University, Tempe, AZ, 85287}
 \curraddr{}
\email{mkcook@asu.edu}

\author[B. Kotschwar]{Brett Kotschwar}
\address{Arizona State University, Tempe, AZ, 85287}
 \curraddr{}
\email{kotschwar@asu.edu}
\thanks{The second author was partially supported by Simons Foundation grant \#359335.}


\date{}

\dedicatory{}

\begin{abstract}
We give a short, direct proof that the full holonomy group of a solution to the Ricci flow is invariant up to isomorphism using the invariance of the reduced
holonomy under the flow.
\end{abstract}

\maketitle

\section{Introduction}

In this note, we consider the holonomy of a family of manifolds $(M, g(t))$ evolving by the Ricci flow.
Recall that the holonomy $\Hol_p(g)$ of a Riemannian manifold $(M, g)$ based at $p\in M$ is the group of endomorphisms of $T_pM$ produced by parallel translation
about piecewise smooth loops based at $p$. The reduced holonomy $\Hol_p^0 (g)$ is the subgroup of $\Hol_p(g)$ produced by parallel translation around
piecewise smooth nullhomotopic loops based at $p$. Any identification of $T_pM$ with $\RR^n$ identifies $\Hol_p(g)$ and $\Hol^0_p(g)$ with subgroups of $\operatorname{O}(n)$ and $\operatorname{SO}(n)$ up to conjugation. When $M$ is connected, the conjugacy classes of these subgroups are independent of the basepoint $p$, and it is common to refer to any fixed representatives $\Hol(g)$ and $\Hol^0(g)$ of these classes as \emph{the} holonomy and reduced holonomy groups of 
$(M, g)$.

It is an old observation of Hamilton \cite{Hamilton4D, HamiltonSingularities} that, under mild hypotheses on the solution $g(t)$, the reduced holonomy $\operatorname{Hol}^0(g(t))$ \emph{cannot expand}: if $\operatorname{Hol}^0(g(t))$ is initially restricted to some subgroup $G\subset \operatorname{SO}(n)$, then it remains so, provided that the solution is complete and of bounded curvature or
otherwise
belongs to some class in which the equation is well-posed. This can be proven with a short argument using only general ingredients. First, one passes to the universal cover and applies Berger's classification. It follows from this classification that one need only verify that Einstein, product, and K\"ahler structures are preserved by the flow. Using the short-time existence theorems of \cite{Hamilton3D, Shi1, Shi2} where needed, one can construct
complete Einstein, product, and K\"ahler solutions to the flow starting from given initial data with those characteristics. The uniqueness results of Hamilton and Chen-Zhu \cite{Hamilton3D, ChenZhu} then imply that
these special solutions are the only solutions within the class with the given initial data. 

In  \cite{KotschwarHolonomy}, the second author later showed that the reduced holonomy $\Hol^0(g(t))$ of a complete solution of uniformly bounded curvature also \emph{cannot contract} and, consequently, that $\Hol^0(g(t))\cong \Hol^0(g(0))$ for all time $t$. However, even with Berger's classification, the problem of non-contraction does not reduce in the same way to one of backward uniqueness of solutions to the Ricci flow. While it is still only necessary to verify that the above three special structures are preserved under the flow, it is not in general possible to solve the parabolic terminal-value problems needed to obtain ``competitor'' solutions with these special structures to compare against the original solution. (The one exception is the Einstein case, in which suitable competitors can be obtained by scaling the initial metric homothetically.) Instead, in \cite{KotschwarHolonomy}, the problem is framed as one of backward uniqueness of the solutions to a related prolonged system which may in turn be treated by the general methods of \cite{KotschwarRFBU, KotschwarFrequency}. This formulation also leads to an alternative proof of the non-expansion of $\Hol^0(g(t))$ (which is closer in spirit to that suggested in \cite{HamiltonSingularities} than the argument sketched above). 

The above results leave open the question of whether the full holonomy $\Hol(g(t))$ is preserved by the Ricci flow. Both $\Hol(g)$ and $\Hol^0(g)$ are fundamentally global invariants of the manifold, however, at each $p$, the reduced holonomy $\Hol_p^0(g)$ is the connected component of the identity in $\Hol_p(g)$, and is determined by the holonomy Lie algebra $\mathfrak{hol}_p(g)$. Since the latter contains the image of the curvature operator $\Rm$ at each point $p$, it is possible to test for the preservation of the reduced holonomy along the flow by studying the kernel of the curvature operator of the solution.

For the full holonomy, which carries information about the global topology of the manifold, there is no such convenient infinitesimal characterization. Assuming the invariance of the reduced holonomy along the flow, the essence of the problem one faces is this: if the lift $(\tilde{M}, \tilde{g}(t))$ to the universal cover $\tilde{M}$ of a solution $(M, g(t))$ to the Ricci flow admits a parallel family of tensors $\tilde{A}(t)$, and if, at some time $t_0$, the tensor $\tilde{A}(t_0)$ descends to a parallel tensor $A(t_0)$ on $M$, then $\tilde{A}(t)$ descends to a smooth parallel family of tensors $A(t)$ on $M$ at all times $t$. One natural strategy to attack this problem is to extend $A(t_0)$ to a family of tensors defined for $t > t_0$ as the solution of an appropriate heat-type equation coupled with the Ricci flow and to argue from the maximum principle that $\nabla_{g(t)}A(t)\equiv 0$; however, this approach does not apply to times $t < t_0$.

In this paper, we show that the invariance (up to isomorphism) of $\Hol(g(t))$ along a solution $g(t)$ to the Ricci flow can nevertheless be obtained from that of $\Hol^0(g(t))$ by a direct argument which applies equally well forward and backward in time. The precise statement of our main theorem is the following.

\begin{theorem}\label{thm:fullholonomy} 
Let $g(t)$ be a solution to the Ricci flow on $M\times [0, T]$ such that $(M, g(t))$ is complete for each $t\in [0, T]$,
and $\sup_{M\times[0, T]}|\Rm|(x, t) < \infty$. Then, for all $q\in M$ and $t\in [0, T]$,
\begin{equation}
\Hol_q(g(t)) = \psi_{t}\circ\Hol_q(g(0))\circ \psi^{-1}_{t},
\end{equation}
where $\psi_{t} \in \operatorname{O}(T_qM, g_q(t))$ satisfies
\[
  \frac{d\psi_{t}}{dt} = \Rc\circ \psi_{t}, \quad \psi_{0} = \operatorname{Id},
\]
for $t\in [0, T]$. Here $\Rc = \Rc(g_q(t)): T_qM \to T_qM$.
\end{theorem}
 
In particular, Theorem \ref{thm:fullholonomy} implies that if a complete Ricci flow with bounded curvature is K\"ahler or splits as a product on any time-slice, it must be K\"ahler or split as a product on all time-slices. In Section \ref{sec:paralleltensors} we show that the complex and product structures will in these cases be independent of time.  The preservation of such structures \emph{forward} in time is of course a well-known property of the Ricci flow. We have already sketched one proof of this fact above; one feature of the argument below (when used in conjunction with \cite{KotschwarHolonomy}) is that it does not make use of the short-time existence and uniqueness of solutions to the equation.

In \cite{KotschwarKaehler}, the second author has also previously shown that the preservation of global K\"ahlerity under the flow follows from the preservation of local K\"ahlerity.  Our proof of Theorem \ref{thm:fullholonomy} is in some sense a generalization of the argument given there. Further results concerning the preservation of the K\"ahler property for Ricci flows with instantaneously bounded curvature can be found in \cite{HuangTam, LiuSzekylhidi}.

\section*{Acknowledgement} The question addressed in this note was brought to the authors' attention by Thomas Leinster and Miles Simon.
The authors wish to thank them for their interest and for subsequent related discussions. 

\section{The preservation of reduced holonomy and a reformulation}
Let $g(t)$ be a smooth solution to the Ricci flow
\begin{equation}\label{eq:rf}
  \pdt g = -2\Rc(g),
 \end{equation}
on $M\times [0, T]$. The holonomy groups $\Hol_p^0(g(t))$ and $\Hol_p(g(t))$ based at a point $p\in M$
are naturally represented as subgroups of the orthogonal group $\operatorname{O}(T_pM, g_p(t))$ relative to the time-varying inner product $g_p(t)$ at $p$. Using Uhlenbeck's trick, we can transform Theorem \ref{thm:fullholonomy} into an equivalent statement for a family of connections whose holonomy groups are instead realized as subgroups of some \emph{fixed} representation
of the orthogonal group.

\subsection{Uhlenbeck's trick}
Fix $t_0 \in [0, T]$ and let $E$ be a vector bundle isomorphic to $TM$ by some fixed isomorphism $\imath_{t_0}:E\to TM$. Using $\imath_{t_0}$, we equip $E$ with the bundle metric
$h$ defined by
\[
  h_q(V, W) = g_q(t_0)(\imath_{t_0}V, \imath_{t_0}W)
\]
for $q \in M$ and $V$, $W\in E_q$. Then, we extend $\imath_{t_0}$ forward and backward in time as the solution to the fiber-wise ODE
\begin{equation}
 \label{eq:uhlenbeck}
     \pd{\imath_t}{t} = \Rc \circ \imath_t,
\end{equation}
where $\Rc = \Rc(g_q(t))$,
to obtain a family $\imath_t: E\to TM$
of bundle isomorphisms for $t\in [0, T]$. With this extension, $\imath_t: (E, h) \to (TM, g(t))$ is in fact a bundle isometry for each $t\in [0, T]$.

Let $\nabla = \nabla^t$ denote the Levi-Civita connection of $g(t)$. We will study the holonomy of $\nabla$ via that of the family of pull-back connections $\delb = \delb^{t}$ on $E$ defined by
\begin{equation}\label{eq:delbtdef}
 \delb^{t}_X V = \imath_{t}^{-1}(\nabla^t_X (\imath_{t}V))
\end{equation}
for $X\in TM$ and $V\in \Gamma(E)$, where $\Gamma(E)$ denotes the space of smooth sections of $E$. The metric $h$ is compatible with the connection $\delb^t$, and the holonomy groups of $\delb^t$ and $\nabla^t$ are related by \begin{equation}\label{eq:holrel}
\Hol_q(\delb^t) = \imath_t^{-1} \circ \Hol_q(\nabla^{t}) \circ \imath_t.
\end{equation}

\subsection{Preservation of reduced holonomy}
As discussed in the introduction, the results in \cite{HamiltonSingularities}, \cite{KotschwarHolonomy} imply that when $(M, g(t))$ is complete and of uniformly bounded curvature, the reduced holonomy group of  $\nabla^t$ is isomorphic to the reduced holonomy of $\nabla^{t_0}$ for any $t$, $t_0\in [0, T]$. While the analytic arguments used in the verification of this fact for $t < t_0$ are fundamentally different from those used for the case $t_0 < t$, the backward-time and forward-time problems
can still be formulated in a unified way in terms of the image of the curvature operator $\Rm = \Rm(g(t))$ of the solution in the bundle of two-forms.

Let $\hol(\nabla^t)$ be the subbundle of  $\End(TM)$ whose fiber $\hol_q(\nabla^t)\subset \mathfrak{so}(T_q M, g_q(t))$ at $q$ 
is the Lie algebra of $\Hol^0_q(\nabla^t)\subset \operatorname{O}(T_q M, g_q(t))$. Let $\mathcal{H}(\nabla^t)\subset \wedge^2T^*M$ be the bundle of two-forms isomorphic
to $\hol(\nabla^t)$ via
the correspondence
\[
        A \in \hol_q(\nabla^t) \mapsto g_q(t)(A\cdot, \cdot) \in \Hc_q(\nabla^t),
\]
and let $\hol(\delb^t)\subset \End(E)$ and $\Hc(\delb^t)\subset\wedge^2 E^*$ denote the analogous families of bundles relative to the connection $\delb^t$. 

In Theorem 1.4 and Appendix A of \cite{KotschwarHolonomy} (compare Theorem 4.1 of \cite{HamiltonSingularities}), it is shown that $\Hc(\nabla^t)$ is time-invariant by first showing that the family of subbundles 
\[
H(t) = (\imath_{t})_{\ast}\Hc(\delb^{t_0})\subset \wedge^2(T^*M)
\]
is a $\nabla^t$-parallel subalgebra
which contains the image of $\Rm(g(t))$. From this, we deduce that $H(t)$ must coincide with $\Hc(\nabla^t)$.  Then, using the definition of $\imath_t$ and the fact that
$H(t)$ contains the image of $\Rm(g(t))$, one verifies by a short calculation that $H(t)$ must actually be independent of time.  Thus, 
\[
\Hc(\nabla^{t}) = H(t) = H(t_0) = \imath_{t_0}^*\Hc(\delb^{t_0}) = \Hc(\nabla^{t_0}).
\]
But, $\Hc(\delb^{t}) = \imath_t^{*}\Hc(\nabla^t)$, so
\[
\Hc(\delb^{t}) = \imath_t^*H(t) = \Hc(\delb^{t_0}),
\]
and $\Hc(\delb^{t})$ is also independent of time. 

Whereas the fibers of $\hol(\nabla^{t})$ are related to the fibers of $\Hc(\nabla^t)$ via the time-dependent isomorphisms
$A \mapsto g(t)(A \cdot, \cdot)$, the fibers of $\hol(\delb^t)$ and $\Hc(\delb^t)$ are related by the \emph{time-independent} isomorphism $A \mapsto h(A\cdot, \cdot)$.
Thus, the fibers $\hol_q(\delb^t)\subset \mathfrak{so}(E_q, h)$ are also independent of time, and it follows that the same is true of $\Hol_q^0(\delb^t) \subset \operatorname{O}(E_q, h)$. 

In terms of the framework we have established, the preservation of reduced holonomy can be restated precisely as follows 

\begin{theorem}[\cite{HamiltonSingularities}, \cite{KotschwarHolonomy}]\label{thm:holonomyalgebra} Let $g(t)$, $\nabla^t$, and $\delb^t$ be as above,
and assume that $(M, g(t))$ is complete and of uniformly bounded curvature for $t\in [0, T]$.
Then
$\hol_q(\delb^t)\subset \mathfrak{so}(E_q, h)$ is independent of $t$
for all $q\in M$.
Hence,
\[
  \Hol^0_q(\delb^t) = \Hol^0_q(\delb^{t_0}), \quad \Hol^0_q(\nabla^t) = \psi_t\circ\Hol^0_q(\nabla^{t_0})\circ \psi^{-1}_t,
\]
for all $q\in M$, $t\in [0, T]$, where $\psi_{t} = \imath_{t}\circ \imath_{t_0}^{-1}$.
\end{theorem}

Theorem \ref{thm:fullholonomy} can now also be restated in terms of the family of connections $\delb^t$.
\begin{theorem}\label{thm:fullholonomy2} 
Provided $(M, g(t))$ is complete and of uniformly bounded curvature,
\begin{equation*}
\Hol_q(\delb^t) = \Hol_q(\delb^{t_0}) 
\end{equation*}
for all $t\in [0, T]$ and $q\in M$.  Consequently,
\begin{equation*}
\Hol_q(\nabla^t) = \psi_t\circ\Hol_q(\nabla^{t_0})\circ \psi^{-1}_t,
\end{equation*}
 where $\psi_{t} = \imath_{t}\circ \imath_{t_0}^{-1}$.
\end{theorem}

\section{Invariance of the full holonomy group}
Given a piecewise smooth curve $\gamma:[a, b]\to M$, we will use $D_s = D_s^t$ to denote the covariant derivative along $\gamma$ induced by $\delb = \delb^t$.  We will temporarily suppress the subscript $t$ on the maps $\imath = \imath_t$. 

The key to the proof of Theorem \ref{thm:fullholonomy2} is the following identity.
\begin{proposition}\label{prop:staystangent}
Let $\gamma: [a, b]\to M$ denote a smooth curve and $V = V(s,t)$ be a smooth family of smooth sections of $E$ along $\gamma = \gamma(s)$ which is parallel along $\gamma$ with respect to $D_s = D_s^t$ for all $t \in [0,T]$. Then $V$ satisfies
\begin{equation*}
D_s \frac{\partial}{\partial t} V = - \imath^{-1} \mathrm{div} \Rm (\dot{\gamma}) \imath V.
\end{equation*}
\end{proposition} 
Here, $\mathrm{div} \Rm $ is the section of $T^*M \otimes \End(TM)$ defined as follows: for any $p \in M$ and $X \in T_pM$, $\mathrm{div} \Rm_p (X) \in \mathrm{End}(T_p M)$ acts on $Y \in T_p M$ by
\begin{equation*}
\mathrm{div} \Rm_p (X)Y = \sum_{l=1}^n \nabla_{e_l} R_p(e_l,X)Y,
\end{equation*}
where $\{e_l\}_{l=1}^{n}$ is a $g(t)$-orthonormal basis of $T_pM$.
Note that $\mathrm{div}\Rm_p(X)\in \hol_p(\nabla^t)$ for any $X\in T_pM$. Indeed, the curvature endomorphisms $R_p \in T^{(3,1)}(T_pM)$ belong to 
$\Hc_p(t) \otimes \hol_p(\nabla^t)$, and so do $\nabla_{X_1}\nabla_{X_2} \cdots \nabla_{X_k} R_p$ for any $X_1, X_2, \ldots, X_k\in T_pM$.

\begin{proof}[Proof of Proposition \ref{prop:staystangent}]
Let $s_0 \in [a, b]$ be fixed. Choose local frames $(E_\alpha)_{\alpha=1}^n$ for $E$ and $(e_i)_{i=1}^n$ for $TM$ on a neighborhood of $\gamma(s_0)$,
and let $\bar{\Gamma}^{\beta}_{i \alpha}$ be the coefficients of $\delb$ in terms of these frames, i.e., $\delb_{e_i} E_{\alpha} = \bar{\Gamma}_{i\alpha}^{\beta}E_{\beta}$. 

First, since $V(\cdot, t)$ is parallel for all $t$,
\begin{equation*}
0 = \frac{\partial}{\partial t} D_s V = \frac{\partial^2 V^\alpha}{\partial s \partial t} E_\alpha + \frac{\partial V^\alpha}{\partial t} \delb_{\dot{\gamma}} E_\alpha + \dot{\gamma}^i V^\alpha \frac{\partial \bar{\Gamma}^\beta_{i \alpha}}{\partial t} E_\beta
\end{equation*}
at any $(s_0, t)$. On the other hand, for $s$ near $s_0$, 
\[
\pd{V}{t}(s, t) = \pd{V^{\alpha}}{t}(s, t)E_{\alpha}(\gamma(s)),
\]
and so
\begin{align}
\nonumber
D_s \frac{\partial V}{\partial t} & = \frac{\partial^2 V^\alpha}{\partial s \partial t} E_\alpha + \frac{\partial V^\alpha}{\partial t} \delb_{\dot{\gamma}} E_\alpha\\
\label{eq:diveq}
& = - \dot{\gamma}^i V^\alpha \frac{\partial \bar{\Gamma}^\beta_{i \alpha}}{\partial t} E_\beta
\end{align}
at $(s_0, t)$.

Now, fix $\alpha$ and temporarily write $X = \imath E_\alpha$. Then we have $\frac{\partial}{\partial t} X = \Rc (X)$,
and, using that
\[
 \frac{\partial \Gamma^k_{ij}}{\partial t}  = \nabla^k R_{ij} -\nabla_i R^k_j - \nabla_j R_i^k, \quad \nabla^k R_{ji} - \nabla_j R_{i}^k = g^{lm}\nabla_l R_{mij}^k,
\]
where $\Gamma_{ij}^k$ denotes the components of $\nabla$ in terms of $\{e_i\}_{i=1}^{n}$,
we see that
\begin{equation*}
\begin{split}
\frac{\partial}{\partial t} (\nabla_{e_i} X) & = \nabla_{e_i} \left( \frac{\partial}{\partial t} X \right) + \left( \frac{\partial \Gamma^k_{ij}}{\partial t} \right) X^j e_k \\
& = \nabla_{e_i}(\Rc(X)) - (\nabla_{e_i}\Rc)(X) - (\nabla_X \Rc)(e_i) + (\nabla^k \Rc)(e_i, X)e_k \\
& =  \Rc (\nabla_{e_i} X ) + \mathrm{div} \Rm (e_i) X.
\end{split}
\end{equation*}
Thus,
\begin{equation*}
\begin{split}
\frac{\partial \bar{\Gamma}^\beta_{i \alpha}}{\partial t} E_\beta & = \frac{\partial}{\partial t} \left( \delb_{e_i} E_\alpha \right) \\
& = \frac{\partial}{\partial t} \left( \imath^{-1} \nabla_{e_i} \imath (E_\alpha) \right) \\ 
& = - \imath^{-1} \circ \Rc (\nabla_{e_i} \imath E_\alpha) + \imath^{-1} \left(\frac{\partial}{\partial t} (\nabla_{e_i} \imath E_\alpha )\right) \\
& =  \imath^{-1} \mathrm{div} \Rm (e_i)\imath E_\alpha.\\
\end{split}
\end{equation*}
Inserting this expression into \eqref{eq:diveq} for $D_s \frac{\partial V}{\partial t}$ completes the proof.
\end{proof}

Next we use Proposition \ref{prop:staystangent} to determine the evolution of parallel transport along a fixed loop. 
\begin{proposition}\label{prop:dtp} Let $q\in M$ and let $\gamma:[0, 1]\to M$ be a piecewise smooth loop with $\gamma(0) = \gamma(1) = q$.
Let $P_{s,t} : E_q \to E_{\gamma(s)}$ be parallel translation along $\gamma$ with respect to $D_s = D_s^t$. Then
\begin{equation*}
\frac{\partial}{\partial t} P_{1,t} = P_{1,t} B
\end{equation*}
for some $B = B(t) \in \hol_q (\delb)$.
\end{proposition}

\begin{proof}
 It suffices to show that
\begin{equation*}
P^{-1}_{1,t} \frac{\partial}{\partial t} P_{1,t} \in \hol_q (\delb).
\end{equation*}
Let $0 = a_0 < a_1 < \ldots < a_k = 1$ be such that $\gamma|_{[a_{i-1}, a_i]}$ is smooth and fix an arbitrary $W\in E_q$. Applying the previous lemma to $V = P_{s, t} W$ on any subinterval $[a_{i-1}, a_i]$, we find that
\begin{equation*}
\begin{split}
\frac{d}{ds} \left( P^{-1}_{s,t} \frac{\partial}{\partial t} P_{s,t} W \right)& =P^{-1}_{s,t} \left( D_s \frac{\partial}{\partial t} P_{s,t}W \right) \\
& = -P^{-1}_{s, t} \left( \imath^{-1} \mathrm{div} \Rm_{\gamma(s)} (\dot{\gamma}) \imath (P_{s,t}W)\right).
\end{split}
\end{equation*}
In other words,
\[
  \frac{d}{ds} \left(P^{-1}_{s,t} \frac{\partial}{\partial t} P_{s,t}\right) =  -P^{-1}_{s,t} \circ \imath^{-1} \circ \mathrm{div} \Rm_{\gamma(s)} (\dot{\gamma}) \circ \imath \circ P_{s,t} \doteqdot A(s, t).
\]

But $\mathrm{div} \Rm_{\gamma(s)} (\dot{\gamma})\in \hol_{\gamma(s)}(\nabla)$ for each $s$, so $\imath^{-1} \circ \mathrm{div} \Rm_{\gamma(s)} (\dot{\gamma}) \circ \imath \in \hol_{\gamma(s)}(\delb)$ for each $s$. Since $\hol(\delb)$ is invariant under parallel translation, it follows that
$A(s, t)\in \hol_q(\delb)$
for all $s\in (a_{i-1}, a_i)$ and $t\in [0, T]$.

Now let $\hol_q^{\perp}(\delb)$ denote the orthogonal complement of $\hol_q(\delb)$ in $\End(E)$ and let $L\in \hol_q^{\perp}(\delb)$ be arbitrary. Then
\[
 F(s) = \left\langle L, P^{-1}_{s,t} \frac{\partial}{\partial t} P_{s,t}\right\rangle_{h_q} = \left\langle P_{s,t} \circ L \circ P_{s,t}^{-1} ,  \frac{\partial}{\partial t} P_{s,t}\circ P_{s,t}^{-1}\right\rangle_{h_{\gamma(s)}}
 \]
 is continuous on $[0, 1]$ and smooth on each interval $(a_{i-1}, a_i)$. For $s$ in any such interval,
 \[
  F^{\prime}(s) =  \left\langle L, A(s, t) \right\rangle_{h_q} = 0.
 \]
Thus $F|_{[a_{i-1}, a_i]}$ is constant for each $i$. 

But $P_{0,t} = P^{-1}_{0,t} = \mathrm{Id}$ for all $t$, so $P^{-1}_{0,t} \frac{\partial}{\partial t} P_{0,t} = 0$
and $F(0) = 0$. Thus $F(s) = 0$ for all $s\in [0, 1]$. Since $L\in \hol_q^{\perp}(\delb)$ was arbitrary, it follows that
\[
B(t) \doteqdot P^{-1}_{1,t} \frac{\partial}{\partial t} P_{1,t}  \in \hol_q (\delb),
\]
completing the proof.
\end{proof}

We will use Proposition \ref{prop:dtp} in conjunction with the following simple fact.
\begin{lemma}\label{lem:lie}
  Suppose $H$ is a Lie subgroup of the Lie group $G$, $B(t)$ is a smooth family of tangent vectors in $T_e H\subset T_eG$ for $t\in [0, T]$, and
$X = X(g, t)$ is the left-invariant
extension of $B(t)$ to $G$ for each $t$. If $\alpha:[0, T]\to G$ is an integral curve of $X$ passing through $a\in H$ at $t= t_0$ then $\alpha(t) \in H$ for all $t\in [0, T]$.
\end{lemma}
\begin{proof}
Since $B(t)\in T_e H$, we may separately form the left-invariant extension $\bar{X}$ of $B(t)$ on $H$ and obtain $\bar{\alpha}:[0, T]\to H$
solving $\bar{\alpha}^{\prime}(t) = \bar{X}(\bar{\alpha}(t), t)$ with $\bar{\alpha}(t_0) = a$. Then, the inclusion $\iota\circ \bar{\alpha}$ of 
$\bar{\alpha}$ into $G$ 
will be an integral curve of $X$ passing through $a$ at $t= t_0$ whose image lies in $H\subset G$. By uniqueness, it must coincide with $\alpha$.
\end{proof}

Now we put the above pieces together to prove Theorem \ref{thm:fullholonomy2}.
\begin{proof}[Proof of Theorem \ref{thm:fullholonomy2}] Fix $q\in M$ and $t_0\in [0, T]$.
We will show that $\Hol_q (\delb^t) \subset \Hol_q (\delb^{t_0})$ for all $t\in [0, T]$. 
Let $\gamma:[0, 1]\to M$ be an arbitrary piecewise-smooth loop based at $q$ and let $P(t) = P_{1,t}: E_q\to E_q$ be
parallel translation along $\gamma$ with respect to the covariant derivative $D_s = D_s^t$ relative to $\delb^t$. By Proposition \ref{prop:dtp}, $\pd{P}{t} = P B$
for some $B = B(t)$ in the time-invariant subalgebra $\hol_q(\delb^t) = \hol_q(\delb^{t_0})\subset \mathfrak{so}(E_q, h)$. For each $t \in [0,T]$, let $X(\cdot, t)$ be the left-invariant extension of $B(t)$ to all of $O(E_q, h)$, given by $X(A, t) = A B(t)$.

Then $P(t)$ is an integral curve of the left-invariant vector field $X(\cdot, t)$, and applying  Lemma \ref{lem:lie} with $H = \Hol_q(\delb^{t_0})$, $G = \operatorname*{O}(E_q, h)$, $\alpha(t) = P(t)$, and $a = P(t_0)$, we obtain that $P(t) \in \Hol_q(\delb^{t_0})$ for all $t$. But $P(t)$ represents $\delb^t$-parallel translation along an arbitrary piecewise smooth loop $\gamma$
based at $q$, so $\Hol_q(\delb^t) \subset \Hol_q (\delb^{t_0})$ for all $t$ as claimed.
\end{proof}

\section{Preservation of parallel tensors}
\label{sec:paralleltensors}
One consequence of Theorem \ref{thm:fullholonomy} is that if $g(t)$ is a complete solution to the Ricci flow of uniformly bounded curvature on $M\times [0, T]$
and $A_0$ is a smooth $\nabla^{t_0}$-parallel tensor for some $t_0 \in [0, T]$, then there is a smooth family
$A(t)$ of  $\nabla^t$-parallel tensors on $M\times [0, T]$ with $A(t_0) = A_0$. 
\begin{corollary}\label{cor:paralleltensors}
 If the tensor field $A_0\in \Gamma(T^{k,l}(M))$ is $\nabla^{t_0}$-parallel for some $t_0\in [0, T]$, then $A(t) = (\imath_t)_{\ast}\imath_{t_0}^{\ast} A_0$
 is $\nabla^t$-parallel for all $t$.
\end{corollary}

Indeed, the section $B_0 = \imath^*_{t_0} A_0$ of the corresponding tensor product of $E$ is $\delb^{t_0}$-parallel, and Theorem \ref{thm:fullholonomy2}
shows that
$\Hol_q(\delb^{t})$ is independent of time for each $q$. So $B_0$ is $\delb^t$-parallel, and $A(t) = (\imath_t)_*B_0$ therefore $\nabla^t$-parallel,
for each $t\in [0, T]$. The family $A(t)$ in Corollary \ref{cor:paralleltensors} can be explicitly described as the solution of the fiberwise linear system
 \begin{align*}
 \frac{\partial}{\partial t} A^{a_1 \dots a_l}_{b_1 \dots b_k}&=  R^c_{b_1} A^{a_1 \dots a_l}_{c b_2  \dots b_k} + \dots +
 R^c_{b_k} A^{a_1 \dots a_l}_{b_1 b_2  \dots c}- R^{a_1}_c A^{c a_2 \dots a_l}_{b_1 \dots b_k} - \dots - R^{a_l}_c A^{a_1 \dots c}_{b_1 \dots b_k},\\
 A(t_0) &= A_0,
 \end{align*}
of ordinary differential equations.

In some cases, the extended family of parallel tensors $A(t)$ will be independent of time. This is true, for example, when the time-slice 
$(M, g(t_0)) = (\tilde{M}\times \hat{M}, \tilde{g}\oplus \hat{g})$ is a Riemannian product,
and $A_0$ is one of the associated complementary orthogonal projections $\tilde{P}$, $\hat{P}\in \End(TM)$. It is also true when the time-slice $(M, g(t_0))$
is K\"ahler and $A_0 = J$ is its complex structure.  We give the argument for these two special cases below.

\subsubsection{Product structures}
Suppose $M = \tilde{M}\times \hat{M}$ and $g(t_0) = \tilde{g}\oplus \hat{g}$ is a Riemannian product. Let $\tilde{P}_0$ and $\hat{P}_0$
denote the orthogonal projections onto the subbundles of $TM$ isomorphic to $T\tilde{M}$ and $T\hat{M}$, respectively. These sections
of $\End(TM)$ are parallel at $t= t_0$, and therefore, by Corollary \ref{cor:paralleltensors}, their extensions defined by
\begin{align}
\begin{split}
\label{eq:peq}
\pd{\tilde{P}}{t} &= \tilde{P} \circ \Rc - \Rc \circ \tilde{P}, \quad
            \tilde{P}(t_0) = \tilde{P}_0,\\
        \pd{\hat{P}}{t} &= \hat{P} \circ \Rc - \Rc \circ \hat{P}, \quad
            \hat{P}(t_0) = \hat{P}_0,
\end{split}
\end{align}
are $\nabla^t$-parallel for each $t$. It also follows directly from \eqref{eq:peq} that $\tilde{P}(t)$ and $\hat{P}(t)$ will remain complementary $g(t)$-orthogonal projections.  But these properties imply that $\tilde{P}\circ \Rc = \Rc \circ \tilde{P}$ and $\hat{P}\circ \Rc = \Rc \circ \hat{P}$ identically on $M\times[0, T]$. So $\pd{\tilde{P}}{t} = \pd{\hat{P}}{t} = 0$,
and the product structure these projections define is constant in time.

\subsubsection{Complex structures}
Similarly, suppose $(M, g(t_0))$ is K\"ahler with complex structure $J_0$. As above, the family $J = J(t)$ defined by 
\[
        \pd{J}{t} = J \circ \Rc - \Rc \circ J, \quad J(t_0) = J_0,
\]
will be $\nabla^t$-parallel and will satisfy $J^2=-\operatorname{Id}$ and $g(J\cdot, J\cdot) = g(\cdot, \cdot)$ for all $t$. However, these conditions
likewise imply that $\Rc\circ J = J\circ \Rc$ for each $t$, and hence that $\pd{J}{t} = 0$, so  $(M, g(t))$ is K\"ahler relative to the \emph{fixed}
complex structure $J_0$ for all $t$. (Compare Section 3 of \cite{KotschwarKaehler}.)


\begin{thebibliography}{99}
\bibliographystyle{amsalpha}

\bibitem[CZ]{ChenZhu} B.-L. Chen and X.-P. Zhu,
{\it Uniqueness of the Ricci flow on complete noncompact manifolds},
J. Diff. Geom. {\bf 74} (2006), no. 1, 119--154. 
 
\bibitem[H1]{Hamilton3D} R. S. Hamilton, 
  {\it Three-manifolds with positive Ricci curvature},
   J. Diff. Geom.  {\bf 17}  (1982), no. 2, 255--306.

\bibitem[H2]{Hamilton4D} R. S. Hamilton,
{\it Four-manifolds with positive curvature operator},
J. Diff. Geom. {\bf 24} (1986), no. 2, 153--179.
   
\bibitem[H3]{HamiltonSingularities} R. S. Hamilton,
{\it The formation of singularities in the Ricci flow}, Surveys in differential geometry, Vol. II 
(Cambridge, MA, 1993), 7--136, Int. Press, Cambridge, MA, 1995. 

\bibitem[HT]{HuangTam} S. Huang and L.-F. Tam,
\textit{K\"ahler-Ricci flow with unbounded curvature},
Amer. J. Math. {\bf 140} (2018), no. 1, 189--220.

\bibitem[K1]{KotschwarRFBU} B. Kotschwar,
\textit{Backwards uniqueness of the Ricci flow}, 
Int. Math. Res. Not. (2010), no. 21. 4064--4097. 


\bibitem[K2]{KotschwarHolonomy} B. Kotschwar,
{\it Ricci flow and the holonomy group},
J. Reine Angew. Math. {\bf 690} (2014), 133--161.  

\bibitem[K3]{KotschwarKaehler} B. Kotschwar,
\textit{K\"ahlerity of shrinking gradient Ricci solitons asymptotic to K\"ahler cones}, J. Geom. Anal., {\bf 28} (2018), no. 3, 2609--2623.   


\bibitem[K4]{KotschwarFrequency} B. Kotschwar,
\textit{A short proof of backward uniqueness for some geometric evolution equations},
Int. J. Math. {\bf 27} (2016), no. 12, 1650102, 17 pp. 
 
\bibitem[LS]{LiuSzekylhidi} G. Liu and G. Sz\'{e}kelyhidi,
\textit{Gromov-Hausdorff limits of K\"ahler manifolds with Ricci curvature bounded below},
{\tt arXiv:1804.08567 [math.DG]}, (2018).
 
\bibitem[S1]{Shi1} W.-X. Shi, 
\textit{Deforming the metric on complete Riemannian manifolds},
J. Diff. Geom. {\bf 30} (1989), no. 1, 223--301. 

\bibitem[S2]{Shi2} W.-X. Shi,
\textit{Ricci flow and the uniformization on complete noncompact K\"ahler manifolds},
J. Diff. Geom. {\bf 45} (1997), no. 1, 94--220. 

\end{thebibliography}
\end{document}